\documentclass{amsart}
\usepackage{amsfonts}
\usepackage{amssymb}
\usepackage{amsmath}

\newtheorem{theorem}{Theorem}[section]
\theoremstyle{plain}

\newtheorem{corollary}[theorem]{Corollary}

\newtheorem{example}{Example}[section]

\newtheorem{lemma}{Lemma}[section]

\newtheorem{remark}[theorem]{Remark}

\numberwithin{equation}{section}
\input{tcilatex}

\begin{document}
\title[Inclusion relations]{$(\alpha ,\beta ,\lambda ,\delta ,m,\Omega
)_{p}- $Neighborhood for some families of analytic and multivalent functions}
\author{Halit ORHAN}
\address{Department of Mathematics, Faculty of Science, Ataturk University,
25240 Erzurum, Turkey.}
\email{orhanhalit607@gmail.com}

\begin{abstract}
In the present investigation, we give some interesting results related with
neighborhoods of $\ p-$valent functions. Relevant connections with some
other recent works are also pointed out.
\end{abstract}

\maketitle

\section{\textbf{INTRODUCTION AND\ DEFINITIONS}}

Let $A$ demonstrate the family of functions $f(z)$ \ of the form

\begin{equation*}
f(z)=z+\sum\limits_{n=2}^{\infty }{a_{n}z^{n}}\text{ \ \ \ }
\end{equation*}

which are analytic in the open unit disk $\mathcal{U}=\left\{ {z\in 
%TCIMACRO{\U{2102} }%
%BeginExpansion
\mathbb{C}
%EndExpansion
:\left\vert z\right\vert <1}\right\} .$

We denote by $\mathcal{A}_{p}(n)$ the class of functions $f(z)$ \ normalized
by 
\begin{equation}
f(z)=z^{p}+\sum\limits_{k=n}^{\infty }{a_{k+p}z^{k+p}}\text{ \ \ \ }(n,p\in 
\mathbb{N}:=\left\{ 1,2,3,...\right\} )
\end{equation}%
which are analytic and $p$-valent in $\mathcal{U}$.

Upon differentiating both sides of (1.1) $m$ times with respect to $z,$ we
have

\begin{equation}
f^{(m)}(z)=\frac{p!}{(p-m)!}z^{p-m}+\sum\limits_{k=n}^{\infty }\frac{(k+p)!}{%
(k+p-m)!}{a_{k+p}z^{k+p-m}}
\end{equation}%
\begin{equation*}
(n,p\in 
%TCIMACRO{\U{2115} }%
%BeginExpansion
\mathbb{N}
%EndExpansion
;m\in 
%TCIMACRO{\U{2115} }%
%BeginExpansion
\mathbb{N}
%EndExpansion
_{0}:=%
%TCIMACRO{\U{2115} }%
%BeginExpansion
\mathbb{N}
%EndExpansion
\cup \{0\};p>m).
\end{equation*}

\bigskip We show by $\mathcal{A}_{p}(n,m)$ the class of functions of the
form (1.2) which are analytic and $p$-valent in $\mathcal{U}$.

The concept of neighborhood for $f(z)\in \mathcal{A}$ was first given by
Goodman \cite{Good}. The concept of $\delta $-neighborhoods $N_{\delta }(f)$
of analytic functions $f(z)\in \mathcal{A}$ was first studied by Ruscheweyh 
\cite{Rusch}. Walker \cite{Walk}, defined a neighborhood of analytic
functions having positive real part. Later, Owa \textit{et al}.\cite%
{OwaSaNuno} generalized the results given by Walker. In 1996, Alt\i nta\c{s}
and Owa \cite{AltOwa} gave $(n,\delta )$-neighborhoods for functions $%
f(z)\in \mathcal{A}$ with negative coefficients. In 2007, $(n,\delta )$%
-neighborhoods for $p$-valent functions with negative coefficients were
considered by Srivastava\ \textit{et al}. \cite{SriOr}, and Orhan \cite{Or}.
Very recently, Orhan \textit{et al.}\cite{OrKadOwa}, introduced a new
definition of $(n,\delta )$-neighborhood for analytic functions $f(z)\in 
\mathcal{A}.$ Orhan \textit{et al.}'s \cite{OrKadOwa} results were
generalized for the functions $f(z)\in \mathcal{A}$ and $f(z)\in \mathcal{A}%
_{p}(n)$ by many author (see, \cite{SagKam,AltunOwaKam,Fra,KayKazOwa}).

In this paper, we introduce the neighborhoods $(\alpha ,\beta ,\lambda ,m,{%
\delta ,\Omega )}_{p}-N(g)$ and $(\alpha ,\beta ,\lambda ,m,{\delta ,\Omega )%
}_{p}-M(g)$ of a function $f^{(m)}(z)$ when $f(z)\in \mathcal{A}_{p}(n).$

\bigskip Using the Salagean derivative operator \cite{Sa}; we can write the
following equalities for the function $f^{(m)}(z)$ given by

\begin{equation*}
D^{0}f^{(m)}(z)=f^{(m)}(z)
\end{equation*}

\begin{equation*}
D^{1}f^{(m)}(z)=\frac{z}{(p-m)}\left( f^{(m)}(z)\right) ^{\prime }=\frac{p!}{%
(p-m)!}z^{p-m}+\sum\limits_{k=n}^{\infty }\frac{(k+p-m)(k+p)!}{(p-m)(k+p-m)!}%
{a_{k+p}z^{k+p-m}}
\end{equation*}

\begin{equation*}
D^{2}f^{(m)}(z)=D(Df^{(m)}(z))=\frac{p!}{(p-m)!}z^{p-m}+\sum\limits_{k=n}^{%
\infty }\frac{(k+p-m)^{2}(k+p)!}{(p-m)^{2}(k+p-m)!}{a_{k+p}z^{k+p-m}}
\end{equation*}

\begin{equation*}
.........................................................
\end{equation*}

\begin{equation*}
D^{\Omega }f^{(m)}(z)=D(D^{\Omega -1}f^{(m)}(z))=\frac{p!}{(p-m)!}%
z^{p-m}+\sum\limits_{k=n}^{\infty }\frac{(k+p-m)^{\Omega }(k+p)!}{%
(p-m)^{\Omega }(k+p-m)!}{a_{k+p}z^{k+p-m}.}
\end{equation*}

We define $\tciFourier :\mathcal{A}_{p}(n,m)\rightarrow \mathcal{A}_{p}(n,m)$
such that

\begin{equation*}
\tciFourier (f^{(m)}(z))=(1-\lambda )\left( D^{\Omega }f^{(m)}(z)\right) +%
\frac{\lambda z}{(p-m)}\left( D^{\Omega }f^{(m)}(z)\right) ^{\prime }\text{ }
\end{equation*}

\begin{equation}
=\frac{p!}{(p-m)!}z^{p-m}+\sum\limits_{k=n}^{\infty }\frac{%
(k+p)!(k+p-m)^{\Omega }(1+\lambda k(p-m)^{-1})}{(p-m)^{\Omega }(k+p-m)!}%
a_{k+p}z^{k+p-m}
\end{equation}

\begin{equation*}
\text{\ }(0\leq \lambda \leq 1;\text{ }\Omega ,m\in 
%TCIMACRO{\U{2115} }%
%BeginExpansion
\mathbb{N}
%EndExpansion
_{0};\text{ }p>m).
\end{equation*}

Let $\tciFourier (\lambda ,m,\Omega )$ denote class of functions of the form
(1.3) which are analytic in $\mathcal{U}$.

For $f,g\in \tciFourier (\lambda ,m,\Omega )$, $f$ said to be $(\alpha
,\beta ,\lambda ,m,{\delta ,\Omega )}_{p}-$neighborhood for $g$ if it
satisfies

\begin{equation*}
\left\vert \frac{e^{i\alpha }\tciFourier ^{\prime }(f^{\left( m\right) }(z))%
}{z^{p-m-1}}-\frac{e^{i\beta }\tciFourier ^{\prime }(g^{(m)}(z))}{z^{p-m-1}}%
\right\vert <\delta \text{ \ \ }(z\in \mathcal{U)}
\end{equation*}

for some\textit{\ }$-\pi \leq \alpha -\beta \leq \pi $ \textit{and} \ \ ${%
\delta >}\frac{p!}{(p-m-1)!}\sqrt{2[1-\cos (\alpha -\beta )]}.$ We show this
neighborhood by $(\alpha ,\beta ,\lambda ,m,{\delta ,\Omega )}_{p}-N(g).$

Also, we say that $f\in (\alpha ,\beta ,\lambda ,m,{\delta ,\Omega )}%
_{p}-M(g)$ if it satisfies

\begin{equation*}
\left\vert \frac{e^{i\alpha }\tciFourier (f^{\left( m\right) }(z))}{z^{p-m}}-%
\frac{e^{i\beta }\tciFourier (g^{(m)}(z))}{z^{p-m}}\right\vert <\delta \text{
\ \ }(z\in \mathcal{U)}
\end{equation*}

for some\textit{\ }$-\pi \leq \alpha -\beta \leq \pi $ \textit{and} \ \ ${%
\delta >}\frac{p!}{(p-m-1)!}\sqrt{2[1-\cos (\alpha -\beta )]}.$

We give some results for functions belonging to $(\alpha ,\beta ,\lambda ,m,{%
\delta ,\Omega )}_{p}-N(g)$ and $(\alpha ,\beta ,\lambda ,m,{\delta ,\Omega )%
}_{p}-M(g).$

\section{Main Results}

Now we can establish our main results.

\begin{theorem}
If $f\in \tciFourier (\lambda ,m,\Omega )$ satisfies%
\begin{equation*}
\ \ \ \sum\limits_{k=n}^{\infty }\frac{(k+p-m)^{\Omega +1}(k+p)!(1+\lambda
k(p-m)^{-1})}{(p-m)^{\Omega }(k+p-m)!}\left\vert e^{i\alpha }{a_{k+p}-e}%
^{i\beta }b{_{k+p}}\right\vert
\end{equation*}%
\begin{equation}
{\leq }{\delta -}\frac{p!}{(p-m-1)!}\sqrt{2[1-\cos (\alpha -\beta )]}
\end{equation}

\textit{for some }$-\pi \leq \alpha -\beta \leq \pi ;$ $p>m$ \textit{and} \
\ ${\delta >}\frac{p!}{(p-m-1)!}\sqrt{2[1-\cos (\alpha -\beta )]},$ \textit{%
then}\ $\ f\in (\alpha ,\beta ,\lambda ,m,{\delta ,\Omega )}_{p}-N(g).$
\end{theorem}

\begin{proof}
\bigskip By virtue of (1.3), we can write%
\begin{equation*}
\left\vert \frac{e^{i\alpha }\tciFourier ^{\prime }(f^{(m)}(z))}{z^{p-m-1}}-%
\frac{e^{i\beta }\tciFourier ^{\prime }(g^{(m)}(z))}{z^{p-m-1}}\right\vert =
\end{equation*}%
\begin{eqnarray*}
&&\left\vert \frac{p!(p-m)}{(p-m)!}e^{i\alpha }+e^{i\alpha
}\sum\limits_{k=n}^{\infty }\frac{(k+p-m)^{\Omega +1}(k+p)!(1+\lambda
k(p-m)^{-1})}{(p-m)^{\Omega }(k+p-m)!}{a_{k+p}z^{k}-}\frac{p!(p-m)}{(p-m)!}%
e^{i\beta }\right. \\
&&\left. -e^{i\beta }\sum\limits_{k=n}^{\infty }\frac{(k+p-m)^{\Omega
+1}(k+p)!(1+\lambda k(p-m)^{-1})}{(p-m)^{\Omega }(k+p-m)!}b{_{k+p}z^{k}}%
\right\vert
\end{eqnarray*}%
\begin{eqnarray*}
&<&\frac{p!}{(p-m-1)!}\sqrt{2[1-\cos (\alpha -\beta )]} \\
&&+\sum\limits_{k=n}^{\infty }\frac{(k+p-m)^{\Omega +1}(k+p)!(1+\lambda
k(p-m)^{-1})}{(p-m)^{\Omega }(k+p-m)!}\left\vert e^{i\alpha }{a_{k+p}-e}%
^{i\beta }b{_{k+p}}\right\vert .
\end{eqnarray*}

If 
\begin{equation*}
\ \ \sum\limits_{k=n}^{\infty }\frac{(k+p-m)^{\Omega +1}(k+p)!(1+\lambda
k(p-m)^{-1})}{(p-m)^{\Omega }(k+p-m)!}\left\vert e^{i\alpha }{a_{k+p}-e}%
^{i\beta }b{_{k+p}}\right\vert
\end{equation*}%
\begin{equation*}
{\leq }{\delta -}\frac{p!}{(p-m-1)!}\sqrt{2[1-\cos (\alpha -\beta )]},
\end{equation*}

then we observe that%
\begin{equation*}
\left\vert \frac{e^{i\alpha }\tciFourier ^{\prime }(f^{(m)}(z))}{z^{p-m-1}}-%
\frac{e^{i\beta }\tciFourier ^{\prime }(g^{(m)}(z))}{z^{p-m-1}}\right\vert
<\delta \text{ \ \ }(z\in \mathcal{U)}\mathit{.}
\end{equation*}

Thus, $f\in (\alpha ,\beta ,\lambda ,m,{\delta ,\Omega )}_{p}-N(g).$ This
evidently completes the proof of Theorem 2.1.
\end{proof}

\begin{remark}
In its special case when%
\begin{equation}
m=\Omega =\lambda =\alpha =0\text{ and }\beta =\alpha ,
\end{equation}
in Theorem 2.1 yields a result given earlier by Altunta\c{s} et al. (\cite%
{AltunOwaKam} p.3, Theorem 1).
\end{remark}

We give an example for Theorem 2.1.

\begin{example}
For given 
\begin{equation*}
g(z)=\frac{p!}{(p-m)!}z^{p-m}+\sum\limits_{k=n}^{\infty }B_{k+p}(\alpha
,\beta ,\lambda ,m,{\delta ,\Omega )}z^{k+p-m}\in \tciFourier (\lambda
,m,\Omega )
\end{equation*}%
\begin{equation*}
\text{\ }(n,p\in \mathbb{N}=\left\{ 1,2,3,...\right\} ;\text{ }p>m;\text{ }%
\Omega ,m\in 
%TCIMACRO{\U{2115} }%
%BeginExpansion
\mathbb{N}
%EndExpansion
_{0})
\end{equation*}

we consider 
\begin{equation*}
f(z)=\frac{p!}{(p-m)!}z^{p-m}+\sum\limits_{k=n}^{\infty }A_{k+p}(\alpha
,\beta ,\lambda ,m,{\delta ,\Omega )}z^{k+p-m}\in \tciFourier (\lambda
,m,\Omega )
\end{equation*}%
\begin{equation*}
(n,p\in \mathbb{N}=\left\{ 1,2,3,...\right\} ;\text{ }p>m;\text{ }\Omega
,m\in 
%TCIMACRO{\U{2115} }%
%BeginExpansion
\mathbb{N}
%EndExpansion
_{0})
\end{equation*}

with 
\begin{equation*}
A_{k+p}=\frac{(p-m)^{\Omega }\{{\delta -}\frac{p!}{(p-m-1)!}\sqrt{2[1-\cos
(\alpha -\beta )]}\}(k+p-m)!(n+p-1)}{(1+\lambda k(p-m)^{-1})(k+p-m)^{\Omega
+1}(k+p-1)!(k+p)^{2}(k+p-1)}e^{-i\alpha }+{e}^{i(\beta -\alpha )}B{_{k+p}.}
\end{equation*}
\end{example}

Then we have that 
\begin{equation*}
\ \sum\limits_{k=n}^{\infty }\frac{(k+p-m)^{\Omega +1}(k+p)!(1+\lambda
k(p-m)^{-1})}{(p-m)^{\Omega }(k+p-m)!}\left\vert e^{i\alpha }A{_{k+p}-e}%
^{i\beta }B{_{k+p}}\right\vert {=}
\end{equation*}%
\begin{equation}
(n+p-1)\left( {\delta -}\frac{p!}{(p-m-1)!}\sqrt{2[1-\cos (\alpha -\beta )]}%
\right) \sum\limits_{k=n}^{\infty }\frac{1}{(k+p-1)(k+p)}.
\end{equation}

Finally, in view of the telescopic series, we obtain 
\begin{eqnarray}
\sum\limits_{k=n}^{\infty }\frac{1}{(k+p-1)(k+p)} &=&\underset{\zeta
\longrightarrow \infty }{\lim }\sum\limits_{k=n}^{\zeta }\left[ \frac{1}{%
k+p-1}-\frac{1}{k+p}\right] \\
&=&\underset{\zeta \longrightarrow \infty }{\lim }\left[ \frac{1}{n+p-1}-%
\frac{1}{\zeta +p}\right]  \notag \\
&=&\frac{1}{n+p-1}.  \notag
\end{eqnarray}

Using (2.4) in (2.3), we get 
\begin{equation*}
\sum\limits_{k=n}^{\infty }\frac{(k+p-m)^{\Omega +1}(k+p)!(1+\lambda
k(p-m)^{-1})}{(p-m)^{\Omega }(k+p-m)!}\left\vert e^{i\alpha }{A_{k+p}-e}%
^{i\beta }B{_{k+p}}\right\vert
\end{equation*}

\begin{equation*}
{=}{\delta -}\frac{p!}{(p-m-1)!}\sqrt{2[1-\cos (\alpha -\beta )]}.
\end{equation*}

\bigskip Therefore, we say that $f\in (\alpha ,\beta ,\lambda ,m,{\delta
,\Omega )}_{p}-N(g).$

Also, Theorem 1 gives us the following corollary.

\begin{corollary}
If $f\in \tciFourier (\lambda ,m,\Omega )$ satisfies%
\begin{equation*}
\sum\limits_{k=n}^{\infty }\frac{(k+p-m)^{\Omega +1}(k+p)!(1+\lambda
k(p-m)^{-1})}{(p-m)^{\Omega }(k+p-m)!}\left\vert \left\vert a{_{k+p}}%
\right\vert {-}\left\vert b{_{k+p}}\right\vert \right\vert
\end{equation*}%
\begin{equation*}
{\leq \delta -}\frac{p!}{(p-m-1)!}\sqrt{2[1-\cos (\alpha -\beta )]}
\end{equation*}

\textit{for some }$-\pi \leq \alpha -\beta \leq \pi $ \textit{and} \ \ ${%
\delta >}\frac{p!}{(p-m-1)!}\sqrt{2[1-\cos (\alpha -\beta )]},$ and $\arg
(a_{k+p})-\arg (b_{k+p})=\beta -\alpha $ $(n,p\in \mathbb{N}=\left\{
1,2,3,...\right\} ;$ $m\in 
%TCIMACRO{\U{2115} }%
%BeginExpansion
\mathbb{N}
%EndExpansion
_{0},$ $p>m),$ \textit{then}\ $\ f\in (\alpha ,\beta ,\lambda ,m,{\delta
,\Omega )}_{p}-N(g).$

\begin{proof}
By theorem (2.1), we see the inequality (2.1) which implies that $f\in
(\alpha ,\beta ,\lambda ,m,{\delta ,\Omega )}_{p}-N(g).$

Since $\arg (a_{k+p})-\arg (b_{k+p})=\beta -\alpha ,$ if $\arg
(a_{k+p})=\alpha _{k+p},$ we see $\arg (b_{k+p})=\alpha _{k+p}-\beta +\alpha
.$ Therefore,%
\begin{equation*}
e^{i\alpha }a_{k+p}-e^{i\beta }b_{k+p}=e^{i\alpha }\left\vert
a_{k+p}\right\vert e^{i\alpha _{k+p}}-e^{i\beta }\left\vert
b_{k+p}\right\vert e^{i(\alpha _{k+p}-\beta +\alpha )}=(\left\vert
a_{k+p}\right\vert -\left\vert b_{k+p}\right\vert )e^{i(\alpha _{k+p}+\alpha
)}
\end{equation*}

implies that 
\begin{equation}
\left\vert e^{i\alpha }a_{k+p}-e^{i\beta }b_{k+p}\right\vert =\left\vert
\left\vert a_{k+p}\right\vert -\left\vert b_{k+p}\right\vert \right\vert .
\end{equation}

Using (2.5) in (2.1) the proof of the corollary is complete.

Next, we can prove the following theorem.
\end{proof}
\end{corollary}

\begin{theorem}
If $f\in \tciFourier (\lambda ,m,\Omega )$ satisfies%
\begin{equation*}
\sum\limits_{k=n}^{\infty }\frac{(k+p-m)^{\Omega }(k+p)!(1+\lambda
k(p-m)^{-1})}{(p-m)^{\Omega }(k+p-m)!}\left\vert e^{i\alpha }a{_{k+p}-e}%
^{i\beta }b{_{k+p}}\right\vert \leq \delta -\frac{p!}{(p-m)!}\sqrt{2[1-\cos
(\alpha -\beta )]}\text{ \ \ }\left( z\in \mathcal{U}\right) .
\end{equation*}

\textit{for some }$-\pi \leq \alpha -\beta \leq \pi ;$ $p>m$ \textit{and} \
\ ${\delta >}\frac{p!}{(p-m)!}\sqrt{2[1-\cos (\alpha -\beta )]}$ then $f\in
(\alpha ,\beta ,\lambda ,m,{\delta ,\Omega )}_{p}-M(g).$

The proof of this teorem is similar with Theorem 2.1.

\begin{corollary}
If $f\in \tciFourier (\lambda ,m,\Omega )$ satisfies%
\begin{equation*}
\sum\limits_{k=n}^{\infty }\frac{(k+p-m)^{\Omega }(k+p)!(1+\lambda
k(p-m)^{-1})}{(p-m)^{\Omega }(k+p-m)!}\left\vert \left\vert a{_{k+p}}%
\right\vert {-}\left\vert b{_{k+p}}\right\vert \right\vert \leq \delta -%
\frac{p!}{(p-m)!}\sqrt{2[1-\cos (\alpha -\beta )]}\text{ \ \ }\left( z\in 
\mathcal{U}\right) .
\end{equation*}

\textit{for some }$-\pi \leq \alpha -\beta \leq \pi ;$ $p>m$ \textit{and} \
\ ${\delta >}\frac{p!}{(p-m)!}\sqrt{2[1-\cos (\alpha -\beta )]}$ and $\arg
(a_{k+p})-\arg (b_{k+p})=\beta -\alpha ,$ then $f\in (\alpha ,\beta ,\lambda
,m,{\delta ,\Omega )}_{p}-M(g).$

Our next result as follows.
\end{corollary}
\end{theorem}

\begin{theorem}
If $f\in (\alpha ,\beta ,\lambda ,m,{\delta ,\Omega )}_{p}-N(g),0\leq \alpha
<\beta \leq \pi ;$ $p>m$ and $\arg (e^{i\alpha }a_{k+p}-e^{i\beta
}b_{k+p})=k\phi ,$ then 
\begin{equation*}
\sum\limits_{k=n}^{\infty }\frac{(k+p-m)^{\Omega +1}(k+p)!(1+\lambda
k(p-m)^{-1})}{(p-m)^{\Omega }(k+p-m)!}\left\vert e^{i\alpha }a{_{k+p}-e}%
^{i\beta }b{_{k+p}}\right\vert {\leq \delta -}\frac{p!}{(p-m-1)!}(\cos
\alpha -\cos \beta ).
\end{equation*}

\begin{proof}
For $f\in (\alpha ,\beta ,\lambda ,m,{\delta ,\Omega )}_{p}-N(g),$ we have%
\begin{eqnarray*}
&&\left\vert \frac{e^{i\alpha }\tciFourier ^{\prime }(f^{\left( m\right)
}(z))}{z^{p-m-1}}-\frac{e^{i\beta }\tciFourier ^{\prime }(g^{(m)}(z))}{%
z^{p-m-1}}\right\vert \\
&=&\left\vert \frac{p!(e^{i\alpha }-e^{i\beta })}{(p-m-1)!}%
+\sum\limits_{k=n}^{\infty }\frac{(k+p-m)^{\Omega +1}(k+p)!(1+\lambda
k(p-m)^{-1})}{(p-m)^{\Omega }(k+p-m)!}(e^{i\alpha }a{_{k+p}-e}^{i\beta }b{%
_{k+p})}z^{k}\right\vert
\end{eqnarray*}%
\begin{equation*}
=\left\vert \frac{p!(e^{i\alpha }-e^{i\beta })}{(p-m-1)!}+\sum%
\limits_{k=n}^{\infty }\frac{(k+p-m)^{\Omega +1}(k+p)!(1+\lambda k(p-m)^{-1})%
}{(p-m)^{\Omega }(k+p-m)!}(e^{i\alpha }a{_{k+p}-e}^{i\beta }b{_{k+p})}%
e^{ik\phi }z^{k}\right\vert <\delta .
\end{equation*}

Let us consider $z$ such that $\arg z=-\phi .$ Then $z^{k}=\left\vert
z\right\vert ^{k}e^{-ik\phi }.$ For such a point $z\in \mathcal{U},$ we see
that%
\begin{eqnarray*}
&&\left\vert \frac{e^{i\alpha }\tciFourier ^{\prime }(f(z))}{z^{p-m-1}}-%
\frac{e^{i\beta }\tciFourier ^{\prime }(g(z))}{z^{p-m-1}}\right\vert \\
&=&\left\vert \frac{p!(e^{i\alpha }-e^{i\beta })}{(p-m-1)!}%
+\sum\limits_{k=n}^{\infty }\frac{(k+p-m)^{\Omega +1}(k+p)!(1+\lambda
k(p-m)^{-1})}{(p-m)^{\Omega }(k+p-m)!}\left\vert e^{i\alpha }a{_{k+p}-e}%
^{i\beta }b{_{k+p}}\right\vert \left\vert z\right\vert ^{k}\right\vert
\end{eqnarray*}%
\begin{equation*}
=\left[ \left( \sum\limits_{k=n}^{\infty }\frac{(k+p-m)^{\Omega
+1}(k+p)!(1+\lambda k(p-m)^{-1})}{(p-m)^{\Omega }(k+p-m)!}\left\vert
e^{i\alpha }a{_{k+p}-e}^{i\beta }b{_{k+p}}\right\vert \left\vert
z\right\vert ^{k}+\frac{p!(\cos \alpha -\cos \beta )}{(p-m-1)!}\right)
^{2}\right.
\end{equation*}%
\begin{equation*}
+\left. \left( \frac{p!(\sin \alpha -\sin \beta )}{(p-m-1)!}\right) ^{2}%
\right] ^{\frac{1}{2}}<\delta .
\end{equation*}

This implies that%
\begin{equation*}
\left( \sum\limits_{k=n}^{\infty }\frac{(k+p-m)^{\Omega +1}(k+p)!(1+\lambda
k(p-m)^{-1})}{(p-m)^{\Omega }(k+p-m)!}\left\vert e^{i\alpha }a{_{k+p}-e}%
^{i\beta }b{_{k+p}}\right\vert \left\vert z\right\vert ^{k}+\frac{p!(\cos
\alpha -\cos \beta )}{(p-m-1)!}\right) ^{2}\text{ }<\delta ^{2},
\end{equation*}

or%
\begin{equation*}
\frac{p!}{(p-m-1)!}(\cos \alpha -\cos \beta )+\sum\limits_{k=n}^{\infty }%
\frac{(k+p-m)^{\Omega +1}(k+p)!(1+\lambda k(p-m)^{-1})}{(p-m)^{\Omega
}(k+p-m)!}\left\vert e^{i\alpha }a{_{k+p}-e}^{i\beta }b{_{k+p}}\right\vert
\left\vert z\right\vert ^{k}<\delta
\end{equation*}

for $z\in \mathcal{U}$. Letting $\left\vert z\right\vert \longrightarrow
1^{-},$ we have that%
\begin{equation*}
\sum\limits_{k=n}^{\infty }\frac{(k+p-m)^{\Omega +1}(k+p)!(1+\lambda
k(p-m)^{-1})}{(p-m)^{\Omega }(k+p-m)!}\left\vert e^{i\alpha }a{_{k+p}-e}%
^{i\beta }b{_{k+p}}\right\vert \leq \delta -\frac{p!}{(p-m-1)!}(\cos \alpha
-\cos \beta ).
\end{equation*}
\end{proof}
\end{theorem}

\begin{remark}
Applying the parametric substitutions listed in (2.2), Theorem 2.4\ and $\ $%
2.6\ would yield a set of known results due to Altunta\c{s} et al. (\cite%
{AltunOwaKam} p.5, Theorem 4; p.6, Theorem 7).
\end{remark}

\begin{theorem}
If $f\in (\alpha ,\beta ,\lambda ,m,{\delta ,\Omega )}_{p}-M(g),0\leq \alpha
<\beta \leq \pi $ and $\arg (e^{i\alpha }a_{k+p}-e^{i\beta }b_{k+p})=k\phi ,$
then%
\begin{equation*}
\sum\limits_{k=n}^{\infty }\frac{(k+p-m)^{\Omega }(k+p)!(1+\lambda
k(p-m)^{-1})}{(p-m)^{\Omega }(k+p-m)!}\left\vert e^{i\alpha }a{_{k+p}-e}%
^{i\beta }b{_{k+p}}\right\vert \leq \delta +\frac{p!}{(p-m-1)!}(\cos \beta
-\cos \alpha ).
\end{equation*}%
The proof of this theorem is similar with Theorem 2.6.

\begin{remark}
Taking $\lambda =\alpha =\Omega =m=0,$ $\beta =\alpha $ and $p=1,$in Theorem
2.8, we arrive at the following Theorem due to Orhan et al.\cite{OrKadOwa}.

\begin{theorem}
\bigskip If $f\in (\alpha ,\delta )-N(g)$ and $\arg (a_{n}-e^{i\alpha
}b_{n})=(n-1)\varphi $ $\ (n=2,3,4,...),$ then 
\begin{equation*}
\sum\limits_{n=2}^{\infty }n\left\vert a_{n}{-e}^{i\alpha }b{_{n}}%
\right\vert \leq \delta +\cos \alpha -1.
\end{equation*}

We give an application of following lemma due to Miller and Mocanu \cite%
{MiMo} (see also, \cite{Jack}).

\begin{lemma}
Let the function%
\begin{equation*}
w(z)=b_{n}z^{n}+b_{n+1}z^{n+1}+b_{n+2}z^{n+2}+...\text{ \ \ \ }(n\in 
\mathcal{U)}
\end{equation*}

be regular in $\mathcal{U}$ with $w(z)\neq 0,$ $(n\in \mathcal{U)}$. If $%
z_{0}=r_{0}e^{i\theta _{0}}$ $(r_{0}<1)$ and $\left\vert w(z_{0})\right\vert
=\max_{\left\vert z\right\vert \leq r_{0}}\left\vert w(z)\right\vert ,$ then 
$z_{0}w^{\prime }(z_{0})=qw(z_{0})$ where $q$ is real and $q\geq n\geq 1.$

Applying the above lemma, we derive
\end{lemma}
\end{theorem}

\begin{theorem}
If $f\in \tciFourier (\lambda ,m,\Omega )$ satisfies 
\begin{equation*}
\left\vert \frac{e^{i\alpha }\tciFourier ^{\prime }(f^{(m)}(z))}{z^{p-m-1}}-%
\frac{e^{i\beta }\tciFourier ^{\prime }(g^{(m)}(z))}{z^{p-m-1}}\right\vert
<\delta (p+n-m)-\frac{p!}{(p-m-1)!}\sqrt{2[1-\cos (\alpha -\beta )]}
\end{equation*}%
\textit{for some }$-\pi \leq \alpha -\beta \leq \pi ;$ $p>m$ \textit{and} \
\ ${\delta >}\left( \frac{p!}{(p+n-m)(p-m-1)!}\right) \sqrt{2[1-\cos (\alpha
-\beta )]},$ then 
\begin{equation*}
\left\vert \frac{e^{i\alpha }\tciFourier (f^{(m)}(z))}{z^{p-m}}-\frac{%
e^{i\beta }\tciFourier (g^{(m)}(z))}{z^{p-m}}\right\vert <\delta +\frac{p!}{%
(p-m)!}\sqrt{2[1-\cos (\alpha -\beta )]}\ \ \left( z\in \mathcal{U}\right) .
\end{equation*}

\begin{proof}
Let us define $w(z)$ by%
\begin{equation}
\frac{e^{i\alpha }\tciFourier (f^{(m)}(z))}{z^{p-m}}-\frac{e^{i\beta
}\tciFourier (g^{(m)}(z))}{z^{p-m}}=\frac{p!}{(p-m)!}(e^{i\alpha }{-e}%
^{i\beta })+\delta w(z).
\end{equation}%
Then $w(z)$ is analytic in $\mathcal{U}$\ and $w(0)=0.$ By logarithmic
differentiation, we get from (2.6) that 
\begin{equation*}
\frac{e^{i\alpha }\tciFourier ^{\prime }(f^{(m)}(z))-e^{i\beta }\tciFourier
^{\prime }(g^{(m)}(z))}{e^{i\alpha }\tciFourier (f^{(m)}(z))-e^{i\beta
}\tciFourier (g^{(m)}(z))}-\frac{p-m}{z}=\frac{\delta w\prime (z)}{\frac{p!}{%
(p-m)!}(e^{i\alpha }{-e}^{i\beta })+\delta w(z)}.
\end{equation*}%
Since 
\begin{equation*}
\frac{e^{i\alpha }\tciFourier ^{\prime }(f^{(m)}(z))-e^{i\beta }\tciFourier
^{\prime }(g^{(m)}(z))}{z^{p-m}\left( \frac{p!}{(p-m)!}(e^{i\alpha }{-e}%
^{i\beta })+\delta w(z)\right) }=\frac{p-m}{z}+\frac{\delta w\prime (z)}{%
\frac{p!}{(p-m)!}(e^{i\alpha }{-e}^{i\beta })+\delta w(z)},
\end{equation*}%
we see that 
\begin{equation*}
\frac{e^{i\alpha }\tciFourier ^{\prime }(f^{(m)}(z))}{z^{p-m-1}}-\frac{%
e^{i\beta }\tciFourier ^{\prime }(g^{(m)}(z))}{z^{p-m-1}}=\frac{p!}{(p-m-1)!}%
(e^{i\alpha }{-e}^{i\beta })+\delta w(z)\left( p-m+\frac{zw\prime (z)}{w(z)}%
\right) .
\end{equation*}%
This implies that 
\begin{equation*}
\left\vert \frac{e^{i\alpha }\tciFourier ^{\prime }(f^{(m)}(z))}{z^{p-m-1}}-%
\frac{e^{i\beta }\tciFourier ^{\prime }(g^{(m)}(z))}{z^{p-m-1}}\right\vert
=\left\vert \frac{p!}{(p-m-1)!}(e^{i\alpha }{-e}^{i\beta })+\delta
w(z)\left( p-m+\frac{zw\prime (z)}{w(z)}\right) \right\vert .
\end{equation*}%
We claim that 
\begin{equation*}
\left\vert \frac{e^{i\alpha }\tciFourier ^{\prime }(f^{(m)}(z))}{z^{p-m-1}}-%
\frac{e^{i\beta }\tciFourier ^{\prime }(g^{(m)}(z))}{z^{p-m-1}}\right\vert
<\delta (p-m+n)-\frac{p!}{(p-m-1)!}\sqrt{2[1-\cos (\alpha -\beta )]}
\end{equation*}%
in $\mathcal{U}$.

Otherwise, there exists a point $z_{0}\in \mathcal{U}$ such that $%
z_{0}w^{\prime }(z_{0})=qw(z_{0})$ (by Miller and Mocanu's Lemma) where $%
w(z_{0})=e^{i\theta }$ and $q\geq n\geq 1.$

Therefore, we obtain that 
\begin{eqnarray*}
\left\vert \frac{e^{i\alpha }\tciFourier ^{\prime }(f^{(m)}(z))}{%
z_{0}^{p-m-1}}-\frac{e^{i\beta }\tciFourier ^{\prime }(g^{(m)}(z))}{%
z_{0}^{p-m-1}}\right\vert &=&\left\vert \frac{p!}{(p-m-1)!}(e^{i\alpha }{-e}%
^{i\beta })+\delta e^{i\theta }\left( p-m+q\right) \right\vert \\
&\geq &\delta \left( p+q-m\right) -\left\vert \frac{p!}{(p-m-1)!}(e^{i\alpha
}{-e}^{i\beta })\right\vert \\
&\geq &\delta \left( p+n-m\right) -\frac{p!}{(p-m-1)!}\sqrt{2[1-\cos (\alpha
-\beta )]}.
\end{eqnarray*}

This contradicts our condition in Theorem 2.11.

Hence, there is no $z_{0}\in \mathcal{U}$ such that $\left\vert
w(z_{0})\right\vert =1.$ This means that $\left\vert w(z)\right\vert <1$ for
all$\ z\in \mathcal{U}.$

Thus, have that 
\begin{eqnarray*}
\left\vert \frac{e^{i\alpha }\tciFourier (f^{(m)}(z))}{z^{p-m}}-\frac{%
e^{i\beta }\tciFourier (g^{(m)}(z))}{z^{p-m}}\right\vert &=&\left\vert \frac{%
p!}{(p-m)!}(e^{i\alpha }{-e}^{i\beta })+\delta w(z)\right\vert \\
&\leq &\frac{p!}{(p-m)!}\left\vert e^{i\alpha }{-e}^{i\beta }\right\vert
+\delta \left\vert w(z)\right\vert \\
&<&\delta +\frac{p!}{(p-m)!}\sqrt{2[1-\cos (\alpha -\beta )]}.
\end{eqnarray*}%
Upon setting $m=0,$ $\alpha =\varphi ,\wp =\tciFourier $ and $\beta =\alpha $
in Theorem 2.11, we have the following corollary given by Sa\u{g}s\"{o}z et
al.\cite{SagKam}.
\end{proof}

\begin{proof}[Corollary 2.12]
If $f\in \wp (\Omega ,\lambda )$ satisfies 
\begin{equation*}
\left\vert \frac{e^{i\alpha }\wp ^{\prime }(f(z))}{z^{p-1}}-\frac{e^{i\beta
}\wp ^{\prime }(g(z))}{z^{p-1}}\right\vert <\delta (p+n)-p\sqrt{2[1-\cos
(\varphi -\alpha )]}
\end{equation*}%
\textit{for some }$-\pi \leq \alpha -\beta \leq \pi ;$\textit{and} \ \ ${%
\delta >}\left( \frac{p}{(p+n)}\right) \sqrt{2[1-\cos (\alpha -\beta )]},$
then%
\begin{equation*}
\left\vert \frac{e^{i\alpha }\wp (f(z))}{z^{p}}-\frac{e^{i\beta }\wp (g(z))}{%
z^{p}}\right\vert <\delta +\sqrt{2[1-\cos (\varphi -\alpha )]}\text{ \ \ }%
\left( z\in \mathcal{U}\right) .
\end{equation*}
\end{proof}
\end{theorem}
\end{remark}
\end{theorem}

\end{document}